\documentclass[12pt, reqno]{amsart}

\usepackage[utf8]{inputenc}
\usepackage[T1]{fontenc}
\usepackage{tikz,tikz-3dplot}
\usepackage{amsmath,amsfonts,amssymb,amsmath,latexsym,mathrsfs}
\usepackage{shuffle}
\usepackage[all]{xy}
\usepackage{stackengine}
 
\usepackage{enumerate}

\usepackage{hyperref}
 
\usepackage{tabu}
\usepackage{tikz-cd} 

\usepackage{comment}
\usepackage{url}
\usepackage{tikz}

\usepackage{xfrac}

\newtheorem{theorem}{Theorem}[section]

\newtheorem{proposition}[theorem]{Proposition}
\newtheorem{definition}[theorem]{Definition}

\newtheorem{corollary}[theorem]{Corollary}

\newtheorem{remark}[theorem]{Remark}

\newtheorem{lemma}[theorem]{Lemma}

\newcommand{\K}{ {\Bbbk} }

 \newcommand{\KP}{{\Bbbk[X]/(P_1)}}
\newcommand{\KPP}{{\Bbbk[X]/(P_2)}}
 \newcommand{\KPn}{{\Bbbk[X]/(P_1^n)}}
 \newcommand{\KPPn}{{\Bbbk[X]/(P_2^n)}}

\begin{document}

\title{on some local rings}

\author{Mohamad \textsc{Maassarani} }
\maketitle
\begin{abstract}
Given two seprable irreducible polynomials $P_1$ and $P_2$ over a filed $\K$. We show that the rings $\KPn$ and $\KPPn$ are isomorphic if and only if their residue fields $\KP$ and $\KPP$ are isomorphic. Partial results in this direction are obtained for the case where the polynomials are not seprable. We note that given a seprable irreducible polynomial $P$ we prove that we have an isomorphism between $\K[X]/(P^n)$ and $(\K[X](P))[Y]/(Y^n)$.
\end{abstract}
\section*{Introduction and main results}
Given two irreducible polynomials $P_1$ and $P_2$ over a filed $\K$. The rings $\KPn$ and $\KPPn$ are local rings, hence if they are isomorphic then their residue fields $\KP$ and $\KPP$ are also isomorphic.\\\\
One can wonder if the converse assertion is true. In that case,  we get that $\KPn$ and $\KPPn$ are isomorphic for all $n\geq1$.\\\\
We show that the converse is true if $P_1$ and $P_2$ are seprable (irreducible) polynomials. In particular, the converse holds for any two polynomails over a perfect field as fields of characteristic $0$ or algebraic extensions of finite fileds. This is show in section \ref{S1}. More precisely, we show that for a seprable irreducible polynomial $P$, we have a $\K$-algebra isomorphism between $\K[X]/(P^n)$ and $(\K[X]/(P))[Y]/(Y^n)$ (theorem \ref{thdec}), and deduce the converse from it (theorem \ref{thres}).\\\\
In section \ref{S2},  We develop criterias (for $P_1$ and $P_2$ irreducible) under wich the existence of certain isomorphisms $f: \KP \to \KPP$ or $f_m: \K[X]/ (P_1^m)\to \K[X]/ (P_2^m)$ imply that $\KPn$ and $\KPPn$ are isomorphic for all $n\geq 1$ (corollary \ref{cr} and theorem \ref{thp1}). These criteria apply to the case where $P_1$ and $P_2$ are not seprable and they are obtained by constructing lifts of $f:\KP \to \KPP$ to morphisms $f_{X,n}: \KPn\to \KPPn$.

\section{Seprable case}\label{S1}
Let $\K$ be a field and $P$ be an irreducible polynomial over $\K$. In this section, we show that if $P$ is seprable $(i.e. \ P'\neq 0)$ then $\K[X]/(P^n)$ is isomorphic as a $\K$-algebra to $(\K[X]/(P))[Y]/(Y^n)$ (theorem \ref{thdec}). From this we deduce that if $P_1$ and $P_2$ are two seprable polynomials, then the local rings $\KPn$ and $\KPPn$ are isomorphic if and only if their residue fields $\KP$ and $\KPP$ are isomophic (theorem \ref{thres}). 

\begin{lemma}
For $Q\in \K[X]$, we have : 
$$ P(X+Q(X))=P(X)+P'(X)Q(X)+R(X)Q(X)^2,$$
for some $R\in \K[X]$.
\end{lemma}
\begin{proof}
Set $P(X)=\sum a_i X^i$. We get $$P(X+Q(X))=\sum a_i (X+Q(X))^i=\sum a_i (X^i+iX^{i-1}Q(X)+R_i(X)Q(X)^2),$$ for some $R_i \in \K[X]$. But $\underset{i>0}{\sum} ia_ix^{i-1}=P'(X)$ and $\sum a_i X^i=P(X)$. This proves the proposition.
\end{proof}

\begin{lemma}
If $P'\neq 0$, then we have an infinite sequence of pairs of polynomials $(Q_0,R_0),(Q_1,R_1),\cdots$ such that for $k\geq 0$ we have : $$ P(X+\overset{k}{\underset{i=1}{\sum}} Q_i(X)P(X)^i)=R_k(X)P(X)^{k+1}.$$
\end{lemma}
\begin{proof}
For $k=0$, the equation is $P(X)=R_0(X)P(X)$. We can take any $Q_0$ and we take $R_0(X)=1$. We will prove the propostion by induction. Assume that the pairs $(Q_0,R_0),\cdots,(Q_n,R_n)$ are constructed, we will constuct $(Q_{n+1},R_{n+1})$. Applying the previous lemma we get that for $U$ and $S$ in $\K[X]$, we have: 
$$  P(U(X)+S(X)P(X)^{n+1})=P(U(X))+P'(U(X))S(X)P(X)^{n+1}+T(X)P(X)^{2n+2}.$$
For $U(X)=X+\overset{n}{\underset{i=1}{\sum}} Q_i(X)P(X)^i$ the eqution is reduced by the induction hypothesis to : 
$$  P(U(X)+S(X)P(X)^{n+1})=(R_n(X)+P'(U(X))S(X))P(X)^{n+1}+T(X)P(X)^{2n+2},$$
and we get a pair $(Q_{n+1},S_{n+1})$ if we find $S(X)$ such that $R_n(X)+P'(U(X))S(X)$ is zero modulo $P(X)$ (i.e. a multiple of $P(X)$). But $P$ is irreducible. Hence $\K[X]/(P)$ is a field and the class of $P'\circ U$ in this field is equal to the class of $P'$ wich is invertible since $P'\neq 0$. So $P'\circ U$ is invertible modulo $P$ and therefore we can find an $S(X)$ such that  $R_n(X)+P'(U(X))S(X)$ is zero modulo $P$. This proves the proposition.
\end{proof}

\begin{proposition}
 If $ P'\neq  0$, then we have an injective $\K$-algebra morphism from the field $\K[X]/(P)$ into $\K[X]/(P^k)$, for $k \geq 1$.
\end{proposition}
\begin{proof}
Let $Q_0,Q_1,\cdots,Q_{k-1}$ be as in the previous proposition and let $\phi_k : \K[X] \to \K[X]/(P^k)$ be the $\K$-algebra morphism given by $$ X\mapsto X+\overset{k-1}{\underset{i=1}{\sum}} Q_i(X)P(X)^i.$$
By the previous proposition $\phi_k$ maps $P$ to $0$ and hence induces an algebra morphism $\bar{\phi}_k : \K[X]/(P) \to \K[X]/(P^k)$. To see that $\bar{\phi}_k$ is injective, notice that $\pi_k\circ \bar{\phi}_k(X)=X$, where $\pi_k$ is the projection $\K[X]/(P^k) \mapsto \K[X]/(P), \ X\mapsto X$.
\end{proof}

\begin{corollary}
If $P' \neq 0$, then the local ring $\K [X] / (P^k)$ contains its residue field $\K [X]/(P)$ as a $\K$-subalgebra.
\end{corollary}

\begin{corollary}
If $P' \neq 0$, then the local ring $\K [X] / (P^k)$ is a $\K [X]/(P)$-algebra.
\end{corollary}

\begin{lemma}
If $P' \neq 0$, then the family $1, P,P^2,\cdots,P^{k-1}$ of $K[X]/(P^k)$ is free over $\K[X]/(P)$. 
\end{lemma}
\begin{proof}
Assume that $a_0 1 + a_1P +\cdots a_{k-1} P^{k-1}=0$ for given $a_i \in \K[X]/(P)$. We want to prove that $a_0=\cdots=a_{k-1}=0$. To see that multiply the equation $P^{k-1}$, we get $a_0P^{k-1}=0.$. Hence $a_0=0$. Since $a_0=0$, multuplying the first equation of the proof by $P^{k-2}$, we get that $a_1P^{k-1}=0$ and hence as before we deduce that $a_1=0$. We show that $a_i=0$ for all the remaining $i$'s by multiplying successively by $P^{k-3},P^{k-4},\dots$.
\end{proof}

\begin{theorem}\label{thdec}
If $P'\neq 0$, then $\K[X]/(P^k)$ is isomorphic as a $\K[X]/(P)$-algebra and as a $\K$-algebra to $$(\K[X]/(P))[Y]/(Y^k).$$ The isomorphism is given by $Y\mapsto P$.
\end{theorem}
\begin{proof}
One has a unique $\K[X]/(P)$-algebra morphism $\psi_k: (\K[X]/(P))[Y]/(Y^k)\to \K[X]/(P^k)$ given by $Y\mapsto P$. This morphism is injective by the previous lemma. The morphism $\psi_k$, is also a $\K$-algebra morphism. The dimension of $\K[X]/(P^k)$ over $\K$ is equal to the degree of $P^k$, hence equal to $k \cdot deg(P)$ ($deg(P)$ is the degree of $P$). The algebra $(\K[X]/(P))[Y]/(Y^k)$ is also of dimension $k\cdot deg(P)$ over $\K$. We therefore have that $\psi_k$ is an injective $\K$-linear map between two vector spaces having the same dimension over $\K$. This proves that $\psi_k$ is an isomorphism. 
\end{proof}
\begin{theorem}\label{thres}
Let $P_1$ and $P_2$ be two irreducible polynomials over $\K$ and $k$ a positive integer. If $P_1$ and $P_2$ are seprable  $(i.e.\ P_i'\neq 0)$, then the local rings $\K[X]/(P_1^k)$ and $\K[X]/(P_2^k)$ are isomorphic if and only if their residue fileds $\K[X]/(P_1)$ and $\K[X]/(P_2)$ are isomorphic.
\end{theorem}
\begin{proof}
If the local rings are isomorphic then the residue fields are isomorphic. Since we assume that $P_i'\neq 0$, we have by the previous theorem that $\K[X]/(P_i^k)$ is isomorphic to $(\K[X]/(P_i))[Y]/(Y^n)$. The "only if" part of the statement follows. 
\end{proof}
\begin{remark}
The condition $P'\neq 0$ is always satisfied if $\K$ is a perfect field as caracteristic $0$ or algebraic extensions of finite fields. Hence, the last two theorems always hold over those fields.
\end{remark}
\begin{proposition}
Let $P_1$ and $P_2$ be two irreducible polynomials over $\K$ and $k$ a positive integer. If $P_1$ and $P_2$ are seprable  $(i.e.\ P_i'\neq 0)$, then the local rings $\K[X]/(P_1^k)$ and $\K[X]/(P_2^k)$ are isomorphic as $\K$-algebras if and only if their residue fileds $\K[X]/(P_1)$ and $\K[X]/(P_2)$ are isomorphic as $\K$-algebras.
\end{proposition}
\begin{proof}
The proof of the previous theorem can be adapted to obtain the proposition. 
\end{proof}

\section{Lifting the isomorphisms}\label{S2}
In this section, $\K$ is a field and $P_1,P_2$ are irreducible polynomials in $\K[X]$. We develop a criteria under wich the existence of certain isomorphisms $f: \KP \to \KPP$ or $f_m: \K[X]/ (P_1^m)\to \K[X]/ (P_2^m)$ imply that $\KPn$ and $\KPPn$ are isomorphic for all $n\geq 1$ (corollary \ref{cr} and theorem \ref{thp1}). These criteria apply to the case where $P_1$ and $P_2$ are not seprable.

\begin{definition}
For $A$ and $B$ two $\K$-algebras, we say that a ring morphism $f:A\to B$ stabilizes $\K$ if there exists a field automorphism $\sigma_f:\K\to \K$ such that $f(a)=\sigma_f(a)$ for $a\in \K$. 
\end{definition}

\begin{proposition}
Let $A$ and $B$ be two finite dimensional algebras over $\K$ and $f:A\to B$ be a ring morphism stabilizing $\K$.
\begin{itemize}
\item[1)] $Im(f)$ is a vector subspace of $B$.
\item[2)] If $f$ is injective then $dim(Im(f))=dim(A)$.
\item[3)] If $f$ is injective and $dim(A)=dim(B)$ then $f$ is an isomorphism.
\item[4)] If $f$ is an isomorphism then $f^{-1}$ stabilizes $\K$ and $\sigma_{f^{-1}}=\sigma_f^{-1}$.
\item[5)] If $f$ is an isomorphism then $dim(A)=dim(B)$.
\item[6)] Let $I$ be a proper ideal of $B$, and let $\pi : B\mapsto B/I$ be the projection. The ring morphism $\pi \circ f$ stabilizes $\K$ and $\sigma_{\pi\circ f}=\sigma_f$.
\item[7)] Let $J$ be an ideal of $A$ lying in the kernel of $f$. The morphism $\bar{f}: A/J \to B$ factorising $f$ stabilizes $\K$ and $\sigma_{\bar{f}}=\sigma_f$. 
\end{itemize}
\end{proposition}
\begin{proof}
This can be proved as for $\K$-algebra morphisms, we only need to use $\sigma_f$ and $\sigma_f^{-1}$.
\end{proof}

We will use the facts in the previous proposition without refering to the proposition.
\begin{proposition}
Let $\sigma : \K \to \K$ be a field automorphism. We have a unique well defined ring automorphism $\sigma^X : \K[X] \to \K[X]$ stabilizing $\K$ given by the data $\sigma^X(X)=X$ and $\sigma_{\sigma^X}=\sigma$.
\end{proposition}
\begin{proof}
This can be readly checked.
\end{proof}

\begin{proposition}\label{P1}
If $f: \KP\to \KPP$ is a ring isomorphism stabilizing $\K$ then :
\begin{itemize}
\item[1)] The degree of $P_1$ is equal to the degree of $P_2$.
\item[2)] There exist a unique polynomial $Q_f\in \K[X]$ of degree less than the degree of $P_1$ (the degree of $P_2$) and greater or equal to $1$ such that $f$ is induced by the ring morphism stabilizing $\K$ $f_X: \K[X]\to \K[X]$ defined by $X\mapsto Q_f  $ and $\sigma_{f_X}=\sigma_f$ i.e. $P\mapsto \sigma_{f}^X(P)\circ Q_f$, where $\sigma_f^X$ is as in the previous proposition.
\item[3)] $\sigma_f^X(P_1)\circ Q_f=S_fP_2$ for a given $S_f \in \K[X]$.
\item[4)] For $P\in K[X]$, if $\sigma_f^X(P)\circ Q_f= SP_2$ for some $S\in \K[X]$ then $P=RP_1$ for some $R\in \K[X]$.  
\item[5)] The morphism $f_X$ maps $(P_1^n)$ into $(P_2^n)$ and hence induces a ring morphism stabilizing $\K$ :  $f_{X,n}:\KPn\to \KPPn$ induced by $P\mapsto \sigma_f^X(P)\circ Q_f$.
\end{itemize}
\end{proposition}
\begin{proof}
Point $1)$ follows from the fact that the dimension of $ \KP$ and $\KPP$ as $\K$-vector spaces are respectively the degree of $P_1$ and the degree of $P_2$. Now, there is a unique polynomial $Q_f$ of $\K[X]$ of degree less than the degree of $P_2$ (degree of $P_1$) such that $f(X_1)=Q_f(X_2)$ where $X_1$ is the class of $X$ in $\KP$ and $X_2$ is the class of $X$ in $\KPP$. This polynomial has a degree greater or equal to $1$ otherwise the image of $f$ will lie in $\K$ ($f$ stabilizes $\K$). For $P\in K[X]$ we have $f(P(X_1))=\sigma_f^X(P)\circ Q_f (X_2)$. This proves $2)$. Since $P_1(X_1)=0$, $f(P_1(X_1))=\sigma_f^X(P)\circ Q_f(X_2)=0$. Hence, $\sigma_f^X(P)\circ Q_f$ lies in the ideal $(P_2)$. This proves $3)$. The statement in $4)$ is equivalent to the injectivity of $f$. Finally $3)$ imples that $\sigma_f^X(P_1^n)\circ Q_f=S_f^nP_2^n$. This proves $5)$. We have proved the proposition.
\end{proof}

\begin{proposition}\label{12}
Let $f: \KP\to \KPP$ be a ring isomorphism stabilizing $\K$ and let $S_f$ and $f_{X,n}$ be as in the previous proposition. For $n>1$, $S_f$ is prime to $P_2$ if and only if  $f_{X,n}:\KPn\to \KPPn$ is an isomorphism.
\end{proposition}
\begin{proof}
We first prove that if $S_f$ is prime to $P_2$ then the map $f_{X,n}$ is an isomorphism.
Assume that $S_f$ is prime to $P_2$ and take $P\in K[X]$ such that its class $\bar{P}$ in $\KPn$ lies in the kernel of $f_{X,n}$, i.e. $\sigma_f^X(P)\circ Q_f=SP_2^n$ for some $S\in K[X]$ ($Q_f$ of the previous proposition). Let $m$ be the largest integer for wich $P_1^m$ divides $P$. We have $P=TP_1^m$ for some $T\in \K[X]$. Combining the last two equation and by applying $3)$ of the previous proposition we get :
$$ \sigma_f^X(P)\circ Q_f=\sigma_f^X(TP_1^m)\circ Q_f=(\sigma_f^X(T)\circ Q_f)S_f^mP_2^m=SP_2^n.$$ 
Assume $m <n$. Since we assumed that $S_f$ is prime to $P_2$, we get that $ \sigma_f^X(T)\circ Q_f=S'P_2^{n-m}$ with $n-m\geq 1$ and hence by $4)$ of the previous proposition $T=RP_1$ for some $R\in \K[X]$. Since $P=TP_1^m$, we now have $P=R_1P_1^{m+1}$. This leads to a contradiction, since $m$ is the largest integer for wich $P_1^m$ divides $P$. Therfore the assumption $m<n$ is false and $P=TP_1^m$ with $m\geq n$. This proves that $P \in (P_1^n)$ and the class $\bar{P}$ of $P$ in $\KPn$ is $0$. We have proved that if $S_f$ is prime to $P$ then $f_{X,n}$ is injective. Since, $\KPn$ and $\KPPn$ have the same dimensions as vector spaces over $\K$ (follows from $1)$ of the previous proposition). We deduce that if $S_f$ is prime to $P$ then the morphism $f_{X,n}$ is an isomorphism. Let us prove the converse. For that assume that $S_f$ is not prime to $P_2$. Hence, (by $3)$ of the previous proposition) and the fact that $P_2$ is irreducible, we have : $$\sigma_f^X( P_1) \circ Q_f= S P_2^m,$$ for some $m>1$ and some $S \in \K[X]$. If $n=2$, we see from the equation that the (nonzero) class of $P_1$ in $\K[X]/(P_1^2)$ lies in the kernel of $f_{X,n}$. For $n>2$, denote by $q_{n,m}$ the quotient of the division of $n$ by $m$. We have $(q_{n,m}+1) m \geq n$ and $q_{n,m}+1<n$. With these conditions, we remark by raising the last equation to the power $q_{n,m}+1$ that the nonzero class of $P_1^{q_{n,m}+1}$ in $\KPn$ lies in the kernel of $f_{X,n}$. We have proved that if $S_f$ is not prime to $P_2$ then the morphism $f_{X,n}$ is not an isomorphism.
\end{proof}
For $f: \KP\to \KPP$ a ring isomorphisms stabilizing $\K$, we will use $Q_f$ and $S_f$ and $\sigma_f^X$ without refrencing.

\begin{proposition}
Let $f: \KP\to \KPP$ be a ring isomorphism stabilizing $\K$.
\begin{itemize}
\item[1)] If $\alpha$ is a root of $P_2$ then $Q_f(\alpha)$ is a root of $\sigma_f^X(P_1)$.
\item[2)] We have a bijection $\{\text{roots of } P_2\}\rightarrow \{\text{roots of } \sigma_f^X(P_1)\}$ given by $\alpha \mapsto Q_f(\alpha)$.
\end{itemize}
\end{proposition}
\begin{proof}
$1)$ follows from the equation $\sigma_f^X(P_1)\circ Q= S_f P_2$ of proposition \ref{P1} and it follows from $1)$ that we have a map $g: \{\text{roots of } P_2\}\rightarrow \{\text{roots of } \sigma_f^X(P_1)\}$ given by $\alpha \mapsto Q_f(\alpha)$. To prove $2)$ we will define an inverse to $g$. Applying $3)$ of propsition \ref{P1} to $f^{-1}$ we get that :
$$(\sigma_f^X)^{-1}(P_2)\circ Q_{f^{-1}}=S_{f^-1}P_1,$$
and hence $$P_2 \circ \sigma_f^X(Q_{f^{-1}})=\sigma_f^X (S_{f^{-1}})\sigma_f^X (P_1).$$
Therefore we have a well defined map  $h:\{\text{roots of } \sigma_f^X(P_1)\} \to \{\text{roots of } P_2\}$ given by $\alpha\mapsto \sigma_f^X(Q_{f^{-1}})(\alpha)$. We will prove that $h$ and $g$ are inverse to each other.  Since $f$ and $f^{-1}$ are inverse to each other, we have :
$$ \sigma_f^X(Q_{f^{-1}})\circ Q_f = X +S_2P_2 \quad \text{and} \quad (\sigma_f^X)^{-1}(Q_f)\circ Q_{f^{-1}}=X+S_1 P_1,$$ 
for some $S_1,S_2 \in \K[X]$. The first equation shows that $hg$ is the identity of $\{\text{roots of } P_2\}$. Composing the second equation by $\sigma_f^X$ we obtain that $gh$ is the identity of $\{\text{roots of } \sigma_f^X(P_1)\}$. We have proved $2)$.
\end{proof}

\begin{proposition}
Let $f: \KP\to \KPP$ be a ring isomorphism stabilizing $\K$. $S_f$ is prime to $P_2$ If and only if $Q_f'\neq 0$.
\end{proposition}
\begin{proof}
We first note that $\sigma_f^X (P_1)$ is irreducible since $P_1$ is irreducible and $\sigma_f^X$ is an automorphism of $\K[X]$. Since that $\sigma_f^X (P_1)$ and $P_2$ have the same degree ( $1)$ of proposition \ref{P1}), that $\sigma_f^X (P_1)$ and $P_2$ have the same number of roots (previous proposition) and that roots of an irreducible polynomial has the same multiplicity, we have by the previous proposition that :
$$\sigma_f^X (P_1)(X)=\overset{m}{\underset{k=1}{\prod}}(X-Q_f(\alpha_k))^r \quad \text{and}\quad P_2(X)=\overset{m}{\underset{k=1}{\prod}}(X-\alpha_k)^r \quad $$
Where $\alpha_1,\dots,\alpha_m$ are the distinct roots of $P_2$ and $r$ is the degree of $P_2$ over $m$. Hence, we have : 
$$(\sigma_f^X (P_1)\circ Q)(X)=\overset{m}{\underset{k=1}{\prod}}(Q(X)-Q_f(\alpha_k))^r.$$ 
The $Q_f(\alpha_k)$'s are distinct (previous proposition). Hence, the multplicity of $\alpha_k$ as a root $\sigma_f^X (P_1)\circ Q$ is equal to the multiplicity of $\alpha_k$ for the factor $(Q(X)-Q_f(\alpha_k))^r$. By proposition \ref{P1}  $\sigma_f^X (P_1)\circ Q=S_fP_2$; we have seen that $ P_2(X)=\overset{m}{\underset{k=1}{\prod}}(X-\alpha_k)^r $ and $P_2$ is irreducible. Hence, $S_f$ is prime to $P_2$ if and 
only $f$ the multiplicity of $\alpha_k$ for the factor $(Q(X)-Q_f(\alpha_k))^r$ is $r$. This last condition is equivalent to $Q'(\alpha_k) \neq 0$. Since $P_2$ is irreducible and $\alpha_k$ is a root of $P_2$ and degree of $Q_f$ is less than the degree of $P_2$, the condition $Q'(\alpha_k) \neq 0$ is equivalent to the condition $Q'\neq 0$. We have proved that $S_f$ is prime to $P_2$ if and only if $Q' \neq 0$.
\end{proof}

\begin{theorem}\label{thp}
Take $n>1$. If $f: \KP\to \KPP$ is a ring isomorphism stabilising $\K$, then $f_{X,n}:\KPn\to \KPPn$ of proposition \ref{P1} is a ring isomorphism stabilising $\K$ if and only if $Q_f'\neq 0$.
\end{theorem}
\begin{proof}
This is obtained by combining the last proposition with proposition \ref{12}
\end{proof}
\begin{corollary}\label{cr}
 If $f: \KP\to \KPP$ is a ring isomorphism stabilising $\K$ such that $Q_f'\neq 0$, then $\KPn\to \KPPn$ are isomorphic for al $n\geq 1$.
\end{corollary}
\begin{theorem}\label{thp1}
Let $f_m: \K[X]/ (P_1^m)\to \K[X]/ (P_2^m)$ be a ring isomorphism stabilizing $\K$ for a given $m\geq 1$.The map $f_m$ maps the class of $X$ onto the class of some $R\in \K[X]$. Let $Q$ be the reminder of the division of $R$ by $P_2$ ($Q$ does not depend on the choice of $R$). If $Q' \neq 0$, then the rings $\KPn$ and $\KPPn$ are isomorphic for all $n\geq 1$. 
\end{theorem}
\begin{proof}
Since $\K[X]/ (P_1^m)$ and $\K[X]/ (P_1^m)$ have a unique maximal ideal, $f_m$ induces a ring isomorphism stabilising $\K$ $f :\KP \to \KPP$ of the residue fields. Now $Q_f=Q$ and hence $Q_f'\neq 0$ and by the previous theorem the morphisms $f_{X,n}: \KP \to \KPn$ are isomorphisms.
\end{proof}

 \end{document}